\documentclass[12pt]{article}
\title{A note on lattice knots}
\author{Sasha Anan'in$^\dagger$, Alexandre Grishkov, Dmitrii Korshunov}

\usepackage{amsmath, amssymb,amsthm}
\usepackage{mathtools}
\usepackage{hyperref}
\usepackage[margin=1in]{geometry}
\usepackage{xcolor}
\usepackage{epigraph}

\newtheorem{definition}{Definition}
\newtheorem{theorem}{Theorem}
\newtheorem{remark}{Remark}
\newtheorem{lemma}{Lemma}
\newtheorem{example}{Example}
\newtheorem*{maintheorem}{Main theorem (informal statement)}
\newtheorem*{maintheoremfinal}{Main theorem}

\numberwithin{definition}{section}
\numberwithin{example}{section}
\numberwithin{remark}{section}

\begin{document}
\maketitle
%\pagenumbering{gobble}

\epigraph{48. A figure is an accident made up of position and habit.\\49. The general directions are six, with the body at the center of diametrical lines.}{{\it Ramon Llull,  Ars Brevis}}

\begin{abstract}
The aim of this note is to share the observation that the set of elementary operations of Turing on lattice knots  \cite{turing} can be reduced to just one type of simple local switches. 
\end{abstract}

\section{Introduction}
The subject of this note is the study of knots in $\mathbb{R}^3$ that can be composed out of unit intervals that are parallel to elements of the standard basis. It can be easily shown that each tame knot can be represented in this form (such representation is often refered to as a {\it cubulation} \cite{BHV}). We will see that any isotopy of knots carries over to a simple combinatorially defined equivalence relation on lattice knots.  Namely,

\begin{maintheorem}
Any two cubulated knots are isotopic if and only if one can be obtained from the other by a sequence of the folowing operations and their inverses as soon as each step does not create self-intersections:

\begin{figure}[h]
\centering
\includegraphics[scale=0.5]{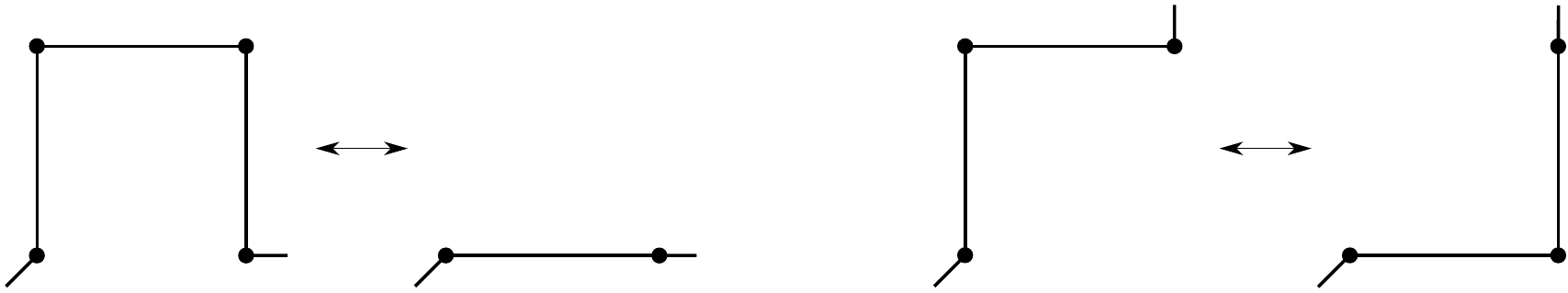}
\caption{Elementary switches}
\label{fig:elementary}
\end{figure}
\end{maintheorem}

It's rather curious that the first-ever survey of the subject of topology, written by Dehn and Heegaard for Klein's Encyklop{\"a}die der mathematischen Wissenschaften \cite{dehn}, was partially devoted to lattice knots. Lacking the modern language of topological spaces, they work with cubulated knots and combinatorial equivalence from the outset.
Also, Alan Turing in his famous article ``Solvable and Unsolvable Problems'' \cite{turing} addresses the problem of algorithmic distinction of cubulated knots. It should also be noted that lattice knots is a subject of interest of biologists as lattice models of polymers \cite{delbruck}, \cite{vanderzande}.
Grid diagrams, which play the key role in a combinatorial description of the knot Floer homology  \cite{grid}, give rise to lattice knots. On the other hand, cubulation is the first step in the construction of a grid diagram of a given knot. At last, we would like to remark that the higher dimensional analogs of cubulated knots also generated some literature \cite{DShSh}.

A version of the main theorem seems to be a folklore knowledge, mentioned as such, e.g. by Turing \cite{turing} and Przytycki.
To the best of the authors' knowledge the first paper where its proof was written up is \cite{verjovsky}. However, all the abovementioned authors require more operations as elementary. In particular, they include doubling (in the sense of our definition below) as a basic operation. The main novelty of our paper is the establishment of the fact that doubling follows from the rest of the operations of \cite{turing} and \cite{verjovsky}.

\subsection{Definitions}

Consider the set $\mathcal N$ obtained as the union of all translations by integer vectors of the union of three coordinate lines:
$$\mathbb Z^3 + \{(x,y,z):x=y=0 \text{ or } x=z=0 \text{ or } z=y=0\}\subset \mathbb R^3$$

\begin{definition}
An embedding $f: S^1\to \mathbb R^3$ such that $f(S^1)\subset \mathcal N$ is called a lattice knot.
\end{definition}

\begin{remark}Lattice knots are also referred to as {\it cubic knots} in the literature \cite{verjovsky}. It is easy to show that any tame knot is isotopic to a lattice knot.
\end{remark}

\begin{figure}[h]
\centering
\includegraphics[scale=1]{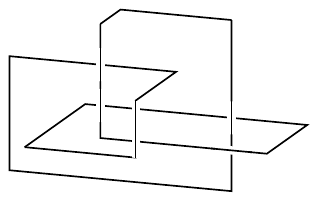}
\caption{Lattice trefoil knot $T$}
\label{fig:trefoil}
\end{figure}

In this paper we will represent a lattice knot by an element of the commutator subgroup of the free group $\mathbb F_3$ in the following way. To each lattice knot we attach a reduced word in the alphabet
$$\{x,y,z,\bar{x},\bar{y},\bar{z}\}$$
which encode linear embeddings of each unit-length segment of a knot.

\begin{example}
\label{ex:trefoil}
A lattice trefoil depicted in Figure \ref{fig:trefoil} in this notation is represented by 
$$T=xxz\bar{y}\bar{y}\bar{x}\bar{x}\bar{x}\bar{z}\bar{z}xxxxzzz\bar{x}\bar{x}y\bar{z}\bar{z}xxx\bar{y}\bar{y}\bar{x}\bar{x}\bar{x}\bar{x}yy$$
\end{example}

\begin{definition}
The  abelianization map $Ab: \mathbb F_3\to \mathbb{Z}^3$ is defined by the formula:
$$Ab: w \mapsto (|w|_x-|w|_{\bar x},|w|_y-|w|_{\bar y},|w|_z-|w|_{\bar z})$$

where $|w|_s$ is the number of occurrences of the symbol $s$ in the reduced word $w$.
\end{definition}

\begin{definition}
A subword of a word $w$ is a set of consecutive symbols of $w$. We will denote it by $w[n..m]$, where $n$ is the index of the first symbol and $m$ is the index of the last. The numbering starts at $1$ and $w[n..m]$ is the empty word if $m<n$.
\end{definition}

Now we are going to recast the definition of a lattice knot in terms of words:

\begin{definition}
\label{kn_condition}
A lattice knot is a reduced word $w$ in the alphabet $\{x,y,z,\bar{x},\bar{y},\bar{z}\}$ such that 
\begin{enumerate}
\item $Ab(w)=0$
\item{No proper subword $w'$ of $w$ has the property $Ab(w')=0$}
\end{enumerate}
\end{definition}

\begin{remark}
These two conditions amount to the closeness and the absence of self-intersections of a knot respectively.
\end{remark}

The following definition introduces the main operation of the paper:

\begin{definition} An operation which takes a subword $v$ in $w$ and substitutes it with $\bar{x}vx$ or $xv\bar{x}$ (analogously for $y$ and $z$) once the result (after a reduction, if necessary) is again a knot, will be called an elementary switch. We will denote it by $w[v\to \bar{x}vx]$, $w[v\to \bar{y}vy]$, and $w[v\to \bar{z}vz]$ etc.
\end{definition}

The typical elementary switch corresponding to the substitution $xx\to zxx\bar z$ is depicted in the following picture:

\begin{center}
\includegraphics[scale=0.6]{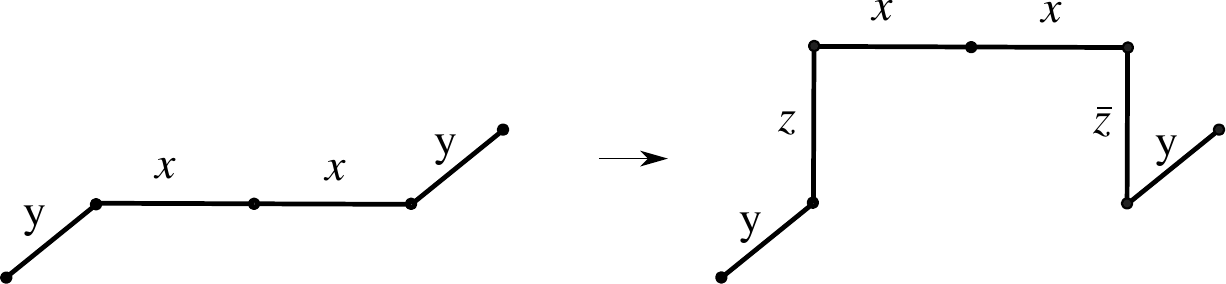}
\end{center}

Geometrically there can be two types of behavior, depending on wheather the result is a reduced word or not. This is illustrated by Figure \ref{fig:elementary}.

\begin{definition} An operation which takes a word and substitutes each $x$ with $xx$ is called a doubling in the $x$-direction. The operation of doubling in directions $y$ and $z$ are defined likewise.
\end{definition}

\begin{remark}
Doubling can be thought of as the same knot positioned in a refined lattice:

\begin{center}
\includegraphics[scale=0.35]{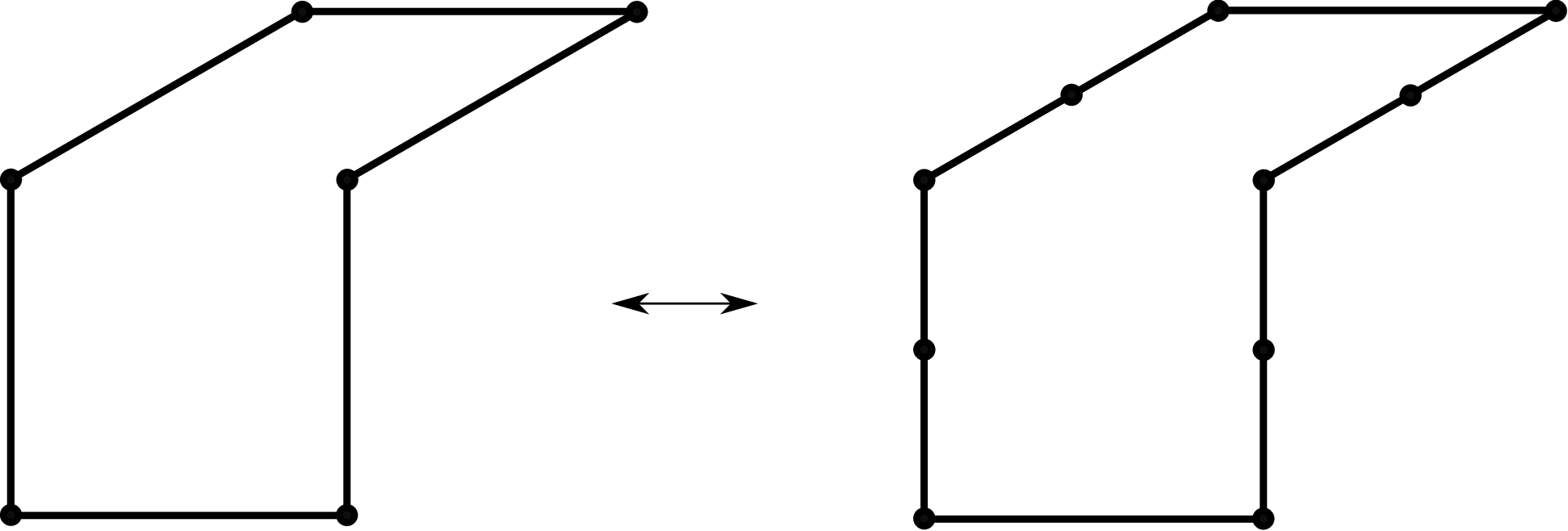}
\label{fig:scaling}
\end{center}

\end{remark}

\section{Doubling}

The purpose of this section is to prove that doubling in the $z$-direction can be achieved by repeated application of elementary switches.
\bigskip
\begin{center}
\includegraphics[scale=0.81]{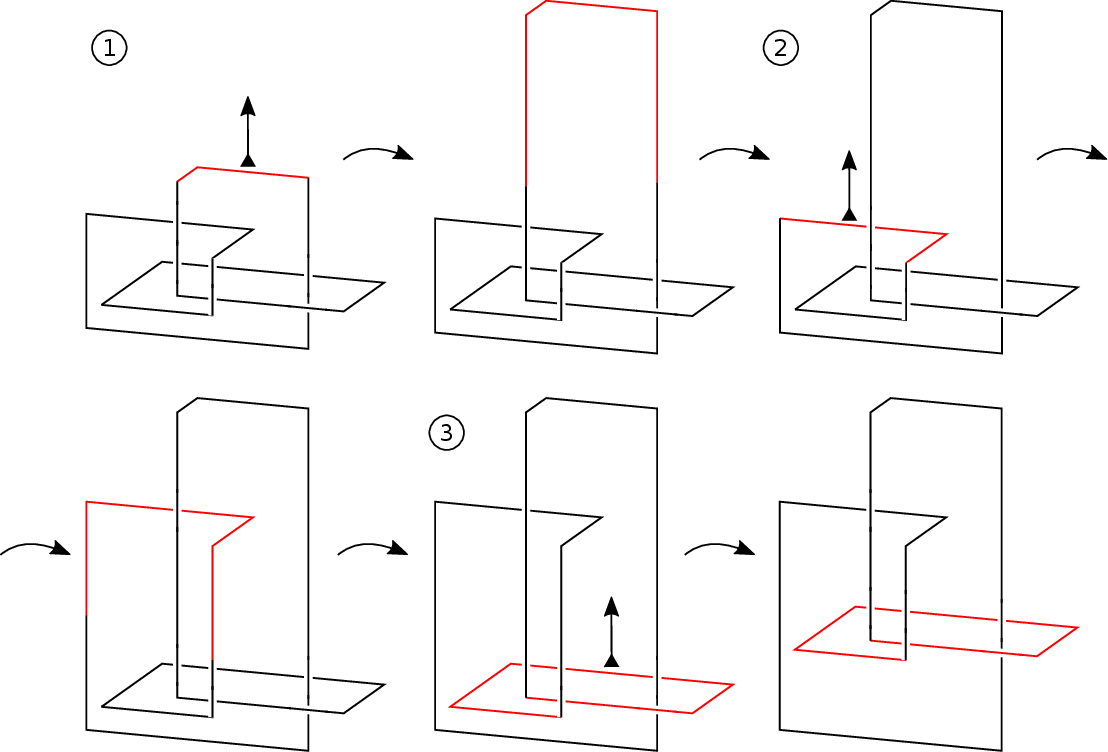}
\label{fig:z-scale}
\end{center}
\bigskip
It is done in stages depicted in the figure above  for the case of the trefoil knot of Example \ref{ex:trefoil}. The idea is very simple. First we lift the uppermost layer of the knot, which is always possible since there are no branches of the knot above it. Then we can  lift the next layer up because there are no horizontal segments between it and the already lifted layer to cause the intersection and eventually we will be able to lift all layers. The rest of the paper will be devoted to the formalization of this picture.

\begin{definition}
The set of all entries $w[m]$ of a word $w$ with the property $|w[1..m-1]|_z=|[1..m]|_z=n$ is called the $n$-th layer of $w$. In other words, it is the set of all entries $w[m]\neq z$, such that $|w[1..m]|_z=n$. Generally a layer is some number of subwords since it need not be connected. 
\end{definition}

The $n$-th layer is simply the set of segments that intersect with a plane parallel to the $xy$-plane and with $z$-coordinate equal to $n$. The picture below shows an example of a knot with $2$ layers, the $1$-layer is showed in red. Note that it is not connected.

\begin{center}
\includegraphics[scale=0.8]{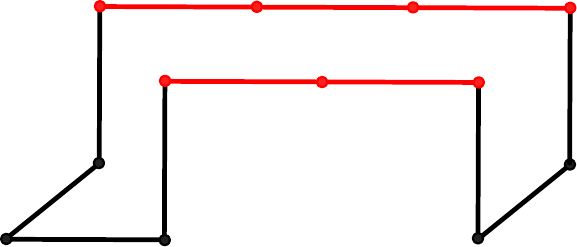}
\end{center}

\begin{example}
In Example \ref{ex:trefoil} the zeroth layer is highlighted in bold:
$$\pmb{xx}z\bar{y}\bar{y}\bar{x}\bar{x}\bar{x}\bar{z}\bar{z}xxxxzzz\bar{x}\bar{x}y\bar{z}\bar{z}\pmb{xxx\bar{y}\bar{y}\bar{x}\bar{x}\bar{x}\bar{x}yy}$$
\end{example}

(-1)-st level is

$$T=xxz\bar{y}\bar{y}\bar{x}\bar{x}\bar{x}\bar{z}\bar{z}\pmb{xxxx}zzz\bar{x}\bar{x}y\bar{z}\bar{z}xxx\bar{y}\bar{y}\bar{x}\bar{x}\bar{x}\bar{x}yy$$

and the second layer is
$$xxz\bar{y}\bar{y}\bar{x}\bar{x}\bar{x}\bar{z}\bar{z}xxxxzzz \pmb{\bar{x}\bar{x}y}\bar{z}\bar{z}xxx\bar{y}\bar{y}\bar{x}\bar{x}\bar{x}\bar{x}yy$$

\begin{definition}
We say that the $n$-th layer has $k$ vacant layers above if all layers of order $n+1,\dots n+k$ are empty.
\end{definition}

\begin{definition}
Elevation of an $n$-th layer by $k$ is the composition of all operations $w[v_i\to z^k v_i \bar z^k]$ (which clearly commute) for each subword $v_i$ that make up the $n$-th layer, followed by reduction. 
\end{definition}

\begin{example}
The first move of Figure \ref{fig:z-scale} is the elevation of the second layer by $3$. That is

\begin{equation*}
    \begin{split}
xxz\bar{y}\bar{y}\bar{x}\bar{x}\bar{x}\bar{z}\bar{z}xxxxzzz {\color{red}\bar{x}\bar{x}y}\bar{z}\bar{z}xxx\bar{y}\bar{y}\bar{x}\bar{x}\bar{x}\bar{x}yy\mapsto 
\\ 
xxz\bar{y}\bar{y}\bar{x}\bar{x}\bar{x}\bar{z}\bar{z}xxxxzzz\pmb{zzz}{\color{red}\bar{x}\bar{x}y}\pmb{\bar{z} \bar{z} \bar {z}}\bar{z}\bar{z}xxx\bar{y}\bar{y}\bar{x}\bar{x}\bar{x}\bar{x}yy
    \end{split}
\end{equation*}

\end{example}

\begin{remark}
The first condition of Definition \ref{kn_condition} is automatically satisfied after an elevation applied, because $z$ and $\bar z$ are added in pairs and the reduction kills pairs.
\end{remark}

The proof of the following lemma is rather obvious geometrically. It is essentially contained in the picture below. 
\begin{center}
\label{fig:intersection}
\includegraphics[scale=0.3]{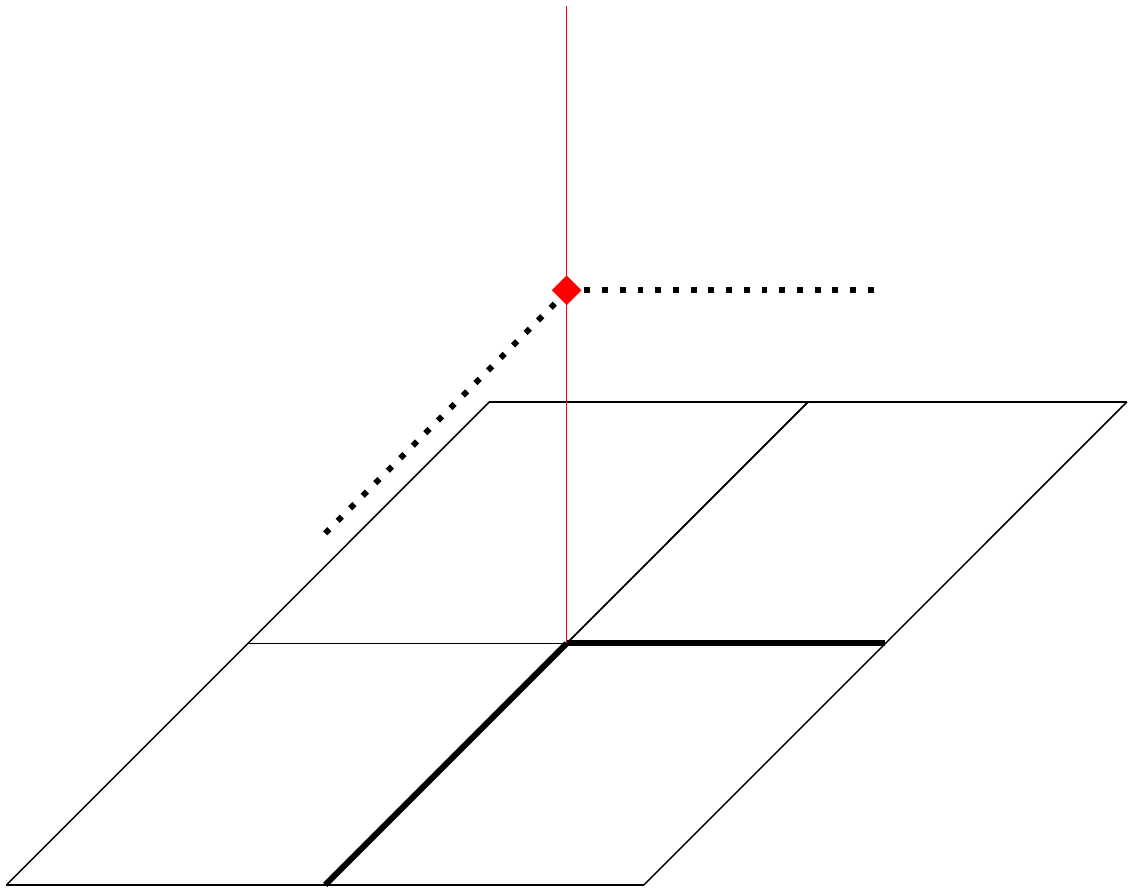}
\end{center}

If after a lifting the $n$-th layer it intersects another branch of the knot, the only possibility is when this other branch consists of two vertical segments in the viscinity of the hypothetical intersection point. But in this case an intersection already took place in the original knot.

For the sake of completeness we include a formal proof.

\begin{lemma}
\label{lift}
If the $n$-th layer of a knot $w$ has one vacant layer above, then the elevation $w'$ of the layer by one is also a knot. Moreover,  the  $n$-th layer of $w'$ becomes vacant, $n+1$-th layer of $w'$ coinsides with the $n$-th layer of $w$ and all other layers remain the same.
\end{lemma}

\begin{proof}
Suppose that after the elevation of the $n$-th layer, the new word, which we denote by $w'$, does not satisfy the second condition of Definition \ref{kn_condition}. That is, there is a proper subword $w'[n'_1..n'_2]$ such that $Ab(w'[n'_1..n'_2])=0$, or, equivalently, there are two different numbers $n'_1,n'_2$ such that $Ab(w'[1..n'_1])=Ab(w'[1..n'_2])$. 

Let us put back all pairs $\bar zz$ and  $z\bar z$ that has been canceled out after reduction. Let us denote by $N_1$ and $N_2$ the new places of $w'(n_1')$ and $w'(n_2')$ in this new (non-reduced) word $W$. 

Let us consider three complementary cases:
\begin{itemize}
\item [Case 1.] 
Suppose $N_1$ and $N_2$ are such that $W[N_1..N_2]$ contains equal number of $z$ and $\bar z$'s added during the elevation. But then the intersection must have existed already in $w$.

\item [Case 2.] Now consider the case when $W[N_1,N_2]$ contains more  added $z$'s than $\bar z$'s. Since $z$ and $\bar z$ are added by the elevation in an alternating manner, there is exactly one more added $z$ than $\bar z$'s. Now, $|W[N_1,N_2]|_z=0$ and hence $|w[n_1,n_2]|_{\bar z}=1$, where $w[n_1,n_2]$ is the subword of $w$ obtained by removing all added $z$'s and $\bar z$'s in  $W[N_1,N_2]$. 

On the other hand $|W[1..N_2]|_z=n+1$ and hence $|W[1..N_1]|_z=n+1$. That means that $W[N_1]=z$ or $\bar z$, otherwise it belonged to the $n+1$'th layer of $w$, which is empty. If $W[N_1]$ were $z$ that has been added by the elevation then $|W[N_1+1..N_2]|_z=0$ --- thus reduced to Case 1. If $W[N_1]$ is $z$ that was already in $w$, then $W[N_1+1]$ is necessarily $z$ that was already in $w$, otherwise the $n+1$'s layer in $w$ wouldn't be empty. And if $W[N_1]=\bar z$ then $|w[n_1+1..n_2]|_z=0$ and we again arrive at contradiction.

\item [Case 3.] When there are one more added $\bar z$'s than $z$'s in $W[n_1,n_2]$, by an argument identical to Case 2 one proves that intersections also does not appear after the elevation. Alternatively, one may use the fact that cyclic permutation of a word gives the same knot and the complement of $w'[n'_1..n'_2]$ satisfying the assumption of Case 3 satisfies the assumption of Case 2.

\end{itemize}

The last statement of the lemma follows from the observation that only the $n$-th and $(n+1)$-th layers are affected by the operation of elevation. Namely, the $n$-th layer dissapears and becomes the $n$-th layer in the new knot.

\end{proof}

Now from Lemma \ref{lift} we can establish the doubling in any direction by induction on the number of layers.

\begin{theorem}
\label{doubling}
Let $w$ be a knot. Then the doubling of $w$ in any direction and $w$ are connected by a sequence of elementary switches. 
\end{theorem}

\begin{proof}
Let $m$ be the maximal number such that the $m$-th layer is nonempty. Clearly it possesses a vacant layer above. Hence we can elevate it $m$ times, ending up with a knot with $2m$ layers with all layers between $m-1$ and $2m$ vacant. We can then elevate the $m-1$-th layer by $m-1$ and so on all the way through to the zeroeth level.
\end{proof}

\section{Isotopy}
A theorem of Hinojosa-Verjovsky-Marcotte \cite[Theorem 1]{verjovsky}, adapted to the terminology of this paper, states

\begin{theorem}
Two lattice knots $w_1$ and $w_2$ are isotopic if and only if one can be obtained from the other by a sequence of elementary switches, doubling and their inverses.
\end{theorem}

Our Theorem \ref{doubling} ensures that doubling follows from elementary switches. Or, in the terminology of \cite{verjovsky} that any (M1) move is a sequence of (M2) moves. Thus, combining two results we arrive at the 

\begin{maintheoremfinal}
Two lattice knots $w_1$ and $w_2$ are isotopic if and only if they can be connected by a sequence of elementary switches.
\end{maintheoremfinal}

\bibliographystyle{plain}
\bibliography{biblo.bib}

\end{document}